\documentclass[pdflatex,sn-mathphys-ay,referee]{sn-jnl}

\usepackage{graphicx}%
\usepackage{multirow}%
\usepackage{amsmath,amssymb,amsfonts}%
\usepackage{amsthm}%
\usepackage{mathrsfs}%
\usepackage[title]{appendix}%
\usepackage[dvipsnames]{xcolor}%
\usepackage{textcomp}%
\usepackage{manyfoot}%
\usepackage{booktabs}%
\usepackage{algorithm}%
\usepackage{algorithmicx}%
\usepackage{algpseudocode}%
\usepackage{listings}%
\usepackage[backgroundcolor=GreenYellow,textsize=tiny,textwidth=15ex]{todonotes}

\theoremstyle{thmstyleone}%
\newtheorem{theorem}{Theorem}[section]
\newtheorem{corollary}[theorem]{Corollary}
\newtheorem{proposition}[theorem]{Proposition}%
\newtheorem{remark}[theorem]{Remark}
\newtheorem{definition}[theorem]{Definition}

\theoremstyle{thmstyletwo}%
\newtheorem{example}{Example}%

\theoremstyle{thmstylethree}%

\newcommand{\eqs}{\stackrel{{\rm sgn}}{=}}
\newcommand{\Fbar}{\overline{F}}
\newcommand{\Gbar}{\overline{G}}
\newcommand{\nG}{G_{\alpha,\beta,\theta}}
\newcommand{\nsG}{G_{\beta,\theta}}
\newcommand{\nsg}{g_{\beta,\theta}}
\newcommand{\nGs}{G_{\alpha_1,\beta_1,\theta_1}}
\newcommand{\nGbar}{\overline{G}_{\alpha,\beta,\theta}}
\newcommand{\nsGbar}{\overline{G}_{\beta,\theta}}

\newcommand{\cdfgll}{K_{\alpha,\beta,\theta}}
\newcommand{\cdfgllalt}{K_{\alpha_1,\beta_1,\theta_1}}
\newcommand{\leqc}{\leq_c}
\newcommand{\leqst}{\leq_{st}}
\newcommand{\leqhr}{\leq_{hr}}
\newcommand{\leqrhr}{\leq_{rh}}
\newcommand{\leqlr}{\leq_{lr}}
\newcommand{\geqc}{\geq_c}
\newcommand{\geqst}{\geq_{st}}
\newcommand{\geqhr}{\geq_{hr}}
\newcommand{\geqrhr}{\geq_{rh}}
\newcommand{\geqlr}{\geq_{lr}}
\newcommand{\IHR}{{\rm IHR}}
\newcommand{\DHR}{{\rm DHR}}
\newcommand{\IOR}{{\rm IOR}}
\newcommand{\DOR}{{\rm DOR}}
\newcommand{\GLL}{{\rm ELL}}
\newcommand{\correction}[2]{\textcolor{red}{#1}\textcolor{OliveGreen}{#2}}

\begin{document}

\title[Order and shape properties of a distorted proportional odds models]{Stochastic orders and shape properties for a new distorted proportional odds model}


\author[1]{\fnm{Idir} \sur{Arab}}\email{idir.bhh@gmail.com}
\equalcont{These authors contributed equally to this work.}

\author[2]{\fnm{Milto} \sur{Hadjikyriakou}}\email{MHadjikyriakou@uclan.ac.uk}
\equalcont{These authors contributed equally to this work.}

\author*[1]{\fnm{Paulo Eduardo} \sur{Oliveira}}\email{paulo@mat.uc.pt}

\affil*[1]{\orgdiv{Department of Mathematics}, \orgname{University of Coimbra}, \country{Portugal}}

\affil[2]{\orgdiv{School of Sciences}, \orgname{University of Central Lancashire}, \country{Cyprus}}



\abstract{Building on recent developments in models focused on the shape properties of odds ratios, this paper introduces two new models that expand the class of available distributions while preserving specific shape characteristics of an underlying baseline distribution. The first model offers enhanced control over odds and log-odds functions, facilitating adjustments to skewness, tail behaviour, and hazard rates. The second model, with even greater flexibility, describes odds ratios as quantile distortions. This approach leads to an enlarged log-logistic family capable of capturing these quantile transformations and diverse hazard behaviours, including non-monotonic and bathtub-shaped rates. Central to our study are the shape relations described through stochastic orders; we establish conditions that ensure stochastic ordering both within each family and across models under various ordering concepts, such as hazard rate, likelihood ratio, and convex transform orders.
}

\keywords{Stochastic order, Odds ratio, Log-logistic, Distortion}


\pacs[MSC Classification]{60E15, 60E05, 62E10}

\maketitle

\section{Introduction}
The development of flexible distribution models is a key pursuit in statistical research, particularly in applications that require detailed control over distributional shape properties, such as survival analysis, reliability engineering, and actuarial science.
{Naturally, the main goal is to introduce families of distributions allowing for fine tuning of some specific characteristic of interest. This is traditionally achieved by adding parameters to established families of distributions or by transforming characterising functionals, such as survival or hazard functions.}
Some recent examples of this approach are \cite{alzaa2013}, \cite{gauss2021} or \cite{gauss2024}{, focusing in practical statistical properties}.
Recent advancements have emphasized extending these frameworks to odds functions, providing a broader and potentially more versatile approach for representing complex real-world data.
{Among the several such models, the Marshall-Olkin family of distributions, introduced in \cite{marshall1997}, and their extensions, have attracted much attention in the literature, due to its simplicity and flexibility. Some recent examples include \cite{Ameeq2024}, \cite{GomesDeniz2024} or \cite{cordeiro2024}.}
{The Marshall-Olkin distributions are built upon some baseline distribution and allow for a relatively simple description of their hazard function, with a strong focus on preserving its monotonicity.}
{However, models with proportional hazard can be found beyond the Marshall-Olkin family. A simple such example is the proportional hazards rate (PHR) model where the survival function is redefined by raising a baseline survival function to some power}.
{It is well-known, increasing hazard rates lead to distributions with finite moments of every order, thus excluding the possibility of models with heavy-tails. On the other hand, decreasing hazard rates models seem less natural in applications. Therefore, there is convenience on broadening the approach to wider families of distributions.}
{An alternative class of distributions, assuming increasingness of the odds function, has been recently studied by \cite{oddspaper} and \cite{Zardasht2024}. This class includes every distribution with increasing hazard rate, and allows for heavy-tailed distributions and some bathtub-shaped hazard rates. Therefore, it is quite natural to explore models where we are interested in the proportionality of the odds function. Examples of model construction based on this idea have been followed by \cite{Bennett1983}, \cite{Collett2023}, \cite{DinseLagakos1983,DinseLagakos1984}, \cite{RossiniTsiatis1996} or \cite{Panja2024}, while structural properties of such models have been studied in \cite{KirmaniGupta2001} or \cite{SanJaya2008}.}

{We suggest two models, both based on preserving monotonicity of the odds function. Being interested in shape properties expressed through the odds, we are naturally driven into studying structural relations expressed by stochastic ordering properties, instead of a more statistical and computational approach.
\correction{}{We note that a similar approach has been studied in \cite{cordeiro2014a}, although concentrated on the control of the tail behaviour of the distributions, and in the statistical properties and estimation aspects.}
The first model \correction{}{we propose}, that we name the \textit{odds-Marshal-Olkin} (oMO), adapts the proportionality idea to the odds function. In fact, the oMO model encompasses the proportionality of the odds and also of the log-odds, providing more flexibility, as applying a logarithmic transformation often leads to an affine relation. The usefulness of this adaptability is illustrated in Example~\ref{ex:Prentice} below. However, the oMO model excludes the already mentioned PHR model, that, in general, produces a more complex odds function. Moreover, the oMO model, although attractive due to its simplicity, allows for limited preservation of interesting stochastic ordering relations or of shape properties. A second and more general construction, that we name \textit{{distorted} odds-Marshal-Olkin model} ({d-oMO}), addresses these difficulties and shows a richer stochastic ordering structure. This broader model defines the odds function by transforming a baseline odds function through the quantile function of an enlarged log-logistic family of distributions (see Definition~\ref{def:GLL} below for details). This relation, means that ordering relations within this new family of distributions translate into the d-oMO distributions, raising the interest in exploring also the properties of this enlarged log-logistic family.}

{The paper is structured as follows. Section~\ref{sec:prelim} provides preliminary definitions and background essential for understanding the proposed models, including a review of stochastic orders and shape properties in distribution families. In Section~\ref{sec:firstmodel}, we introduce the {oMO} model, with a focus on its properties and stochastic comparisons with the baseline distribution. Section~\ref{sec:distmodel} extends the model by {defining} a distorted odds ratio model{, the d-oMO model}, leveraging additional parameters to further control distributional shape. Finally, in Section~\ref{sec:GLL}, we explore the enlarged log-logistic distribution, detailing its implications for odds and hazard rate behaviours.}

\section{Preliminaries and basic definitions}
\label{sec:prelim}
We shall represent by $X$, $F$ and $f$ the baseline random variable, its cumulative and density functions (that we will be assuming to exist), respectively. Analogously, $Y$, $G$ and $g$, possibly with some subscripts to denote parameters, will represent the new models to be studied.
Moreover, survival functions are represented as $\Fbar(x)=1-F(x)$ or $\Gbar(x)=1-G(x)$. We shall refer to the random variables or to their distribution functions as is more convenient. In fact, the characterisations we will be discussing depend only on the distribution, so the random variables will appear only as a convenience.
We recall the usual notions {which were} briefly mentioned in the Introduction. Given a distribution function $F$, its hazard rate and reversed hazard rate are denoted with $h_F(x)=\frac{f(x)}{\Fbar(x)}$ and $\widetilde{h}_F(x) = \frac{f(x)}{F(x)}$ respectively, its odds function with $\Lambda_F(x)=\frac{F(x)}{\Fbar(x)}{=\frac1{\Fbar(x)}-1}$, and its odds rate with $\lambda_F(x)=\Lambda_F^\prime(x)=\frac{f(x)}{\Fbar^2(x)}$.
While the monotonicity of the hazard rate function has been extensively studied in the literature, for the odds function, which is always increasing, the interest relies on its growth rate, characterised by monotonicity of $\lambda_F$. These functions may be used to define some classes of distributions.


\begin{definition}
\label{def:classes}
We say that $X$ or $F$ have
    \begin{enumerate}
    \item
     increasing (decreasing) hazard rate, represented by $F\in\IHR$ ($F\in\DHR)$, if $h_F$ is increasing (decreasing);
    \item
   increasing (decreasing) odds rate, represented by $F\in\IOR$ $(F\in\DOR)$, if $\lambda_F$ is increasing (decreasing);
    \item
    convex (concave) log-odds if $\log\Lambda_F(x)$ is convex (concave).
    \end{enumerate}
\end{definition}
The \IHR\ and \DHR\ families are well-known in the literature, while the \IOR\ family has been receiving less attention. Some properties of the \IOR\ class are studied in a systematic way in \cite{oddspaper}. The \DOR\ family is only briefly mentioned in \cite{ineq-order2024}, and, also recently discussed in \cite{superPareto2024}, although with a different terminology. Note that $F\in\IOR$ $(F\in\DOR)$ is equivalent to the odds {function} $\Lambda_F$ being convex (concave), so these odds {rate} classes describe a shape property of the corresponding distributions.

We now recall some common stochastic order notions that will be considered later.
\begin{definition}
\label{def:order1}
Consider two distribution functions $F_1$ and $F_2$, with densities $f_1$ and $f_2$, respectively. We say that $F_1$ is smaller than $F_2$
    \begin{enumerate}
    \item
    in the usual stochastic order, denoted as $F_1\leqst F_2$, if $\Fbar_1(x)\leq \Fbar_2(x)$, for every $x\in\mathbb{R}$;
    \item
    in the hazard rate order, denoted as $F_1\leqhr F_2$, if $h_{F_1}(x)\geq h_{F_2}(x)$, for every $x\in\mathbb{R}$;
    \item
    in the reversed hazard rate order, denoted as $F_1\leqrhr F_2$, if $\widetilde{h}_{F_1}(x)\geq \widetilde{h}_{F_2}(x)$, for every $x\in\mathbb{R}$;
    \item
    in the likelihood rate order, denoted as $F_1\leqlr F_2$, if $\frac{f_2(x)}{f_1(x)}$, is increasing.
    \item
    in the dispersive order, denoted as $F_1\leq_{disp}F_2$, if $F_2^{-1}\circ F_1(x)-x$ increases in $x$.
    \end{enumerate}
\end{definition}

Given the alternative expression for the odds function, the following statement is straightforward.
\begin{proposition}
\label{prop:oddsa}
Given two distribution functions $F_1$ and $F_2$, $F_1\leqst F_2$ if and only if $\Lambda_{F_1}(x)\geq\Lambda_{F_2}(x)$.
\end{proposition}

The classes mentioned in Definition~\ref{def:classes} defined by the monotonicity of the hazard or the odds rate may be characterised via a different type of stochastic order, namely the convex transform order which involves a shape restriction on the transformation that maps one distribution to the one being compared.

\begin{definition}[\cite{vanzwet1964}]
\label{def:vanZwet}
Given two distribution functions $F_1$ and $F_2$, we say that $F_1$ is smaller than $F_2$ in the convex transform order, represented by $F_1\leqc F_2$, if $F_2^{-1}\circ F_1$ is convex.
\end{definition}
Let us now fix, for the sequel, two reference distributions: the standard exponential, with distribution function $\mathcal{E}(x)=1-e^{-x}$, and the {standard} log-logistic, with distribution function $\mathcal{L}(x)=1-\frac{1}{x+1}=\frac{x}{x+1}$.
It is well-known that $F\in\IHR$ $(F\in\DHR)$ if and only if $F\leqc\mathcal{E}$ $(F\geqc\mathcal{E})$. Analogously, as referred in \cite{oddspaper}, it is easily seen that $F\in\IOR$ $(F\in\DOR)$ if and only if $F\leqc\mathcal{L}$ $(F\geqc\mathcal{L})$.

For {a systematic study of} properties of {the} stochastic orders {defined above, and a number of other interesting stochastic order relations,} and relations among them, we refer the interested reader to the monographs \cite{shaked2007} or \cite{marshall2007}.

\section{{The odds-Marshall-Olkin model}}
\label{sec:firstmodel}
The study of the growth rate of the odds function is fundamental in the characterisation of distribution families that maintain specific shape properties such as the IOR, particularly in reliability and survival analysis. In this context, we introduce a modified proportional odds model that leverages the properties of the IOR and log-odds convexity to create new distribution families. This approach
{extends the Marshall-Olkin method to the construction of families of distributions to} the broader proportional odds framework.

\begin{definition}
\label{def:model1}
Let $\beta,\theta>0$. Given a baseline distribution function $F$, we define the \textit{odds-Marshall-Olkin} (oMO) distribution function $\nsG$ by
\begin{equation}
\label{eq:def1}
\Lambda_{\nsG}(x)=\beta\Lambda_F^\theta(x)=\beta \left(\frac{F(x)}{\Fbar(x)}\right)^\theta.
\end{equation}
\end{definition}
It is obvious that $G_{\beta,1}$ has odds {function} $\Lambda_{G_{\beta,1}}$ proportional to $\Lambda_F$, while for $G_{1,\theta}$ we have that $\log\Lambda_{G_{1,\theta}}(x)=\theta\log\Lambda_{F}(x)$, that is, (\ref{eq:def1}) covers the case of a model with proportional log-odds.
{The classical Marshal-Olkin model, for which there exists a huge literature, is obtained by taking $\theta=1$}.
{It is also clear from (\ref{eq:def1}) that the {oMO} model encompasses affine relations of the odds function with respect to the baseline distribution, thus going beyond the strict proportionality.}

Taking into account that $\Lambda_{\nsG}(x)=\frac{\nsG(x)}{\nsGbar(x)}=\frac{1}{\nsGbar(x)}-1$, it follows easily that, for each $x\in\mathbb{R}$,
\begin{equation}
\label{def:Gbar}
\nsG(x)=\frac{\beta F^\theta(x)}{\beta F^\theta(x)+\Fbar^\theta(x)},
\qquad\text{and}\qquad
\nsGbar(x)=\frac{\Fbar^\theta(x)}{\beta F^\theta(x)+\Fbar^\theta(x)}.
\end{equation}

{The} following shape characterisation\correction{}{, for $\theta=1$,} is a straightforward consequence of the convexity properties of the function $\frac{\beta x}{1-(1-\beta)x}$ when $x\in[0,1]$.
\begin{proposition}
\label{prop:conv}
If $F$ is concave, then $G_{\beta,1}$ for $\beta\geq 1$ is also concave. If $F$ is convex, then $G_{\beta,1}$ for $\beta\leq 1$ is also convex.
\end{proposition}

{It is easily seen} that the corresponding transformation for $\theta\ne1$ and general $\beta>0$ is 
neither convex nor concave, so no conclusion about the convexity of $\nsG$ can be drawn.

From (\ref{def:Gbar}), the density and hazard rate functions for $\nsG$ are easily obtained:
\begin{equation}
\label{eq:gdens}
\nsg(x)=\beta \theta f(x)\frac{F^{\theta-1}(x)\Fbar^{\theta-1}(x)}{(\beta F^\theta(x)+\overline{F}^\theta(x))^2},
\end{equation}
and
\begin{equation}
\label{eq:hazard}
h_{\nsG} (x)=\beta \theta h_F(x)\frac{F^{\theta-1}(x)}{\beta F^\theta(x)+\Fbar^\theta(x)}
   =\beta \theta h_F(x) \left(\frac{F(x)}{\Fbar(x)}\right)^{\theta-1}\frac{\nsGbar(x)}{\Fbar(x)}.
\end{equation}

{The effect of the parameters on the density of $\nsG$ is illustrated in Figure~\ref{fig1-2}. The parameter $\theta$ concentrates the distribution while creating lighter tails, whereas $\beta$ shifts the concentration region to the left, skewing the distribution to the right and making the tail slightly lighter.}
{Obviously, the interplay between the parameters provides control over spread and tail behaviour, offering an increased adaptability in applications.}
\begin{figure}
\centering
\includegraphics[scale=.4]{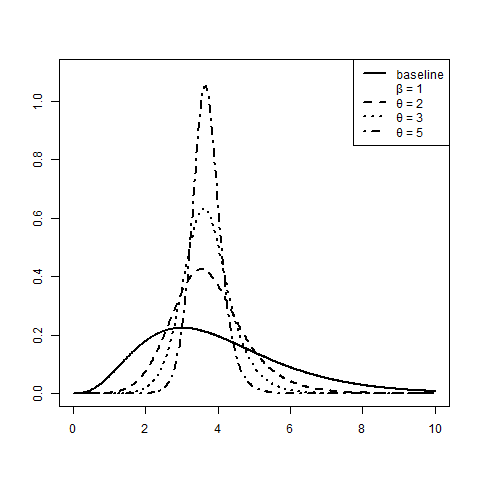}\includegraphics[scale=.4]{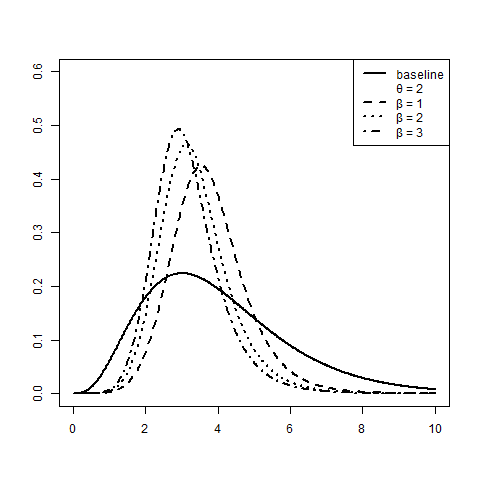}
\caption{Densities of the baseline distribution $\Gamma(4,1)$ and $\nsG$.}
\label{fig1-2}
\end{figure}

Defining $T_{\beta,\theta}(x)=\beta \theta \frac{x^{\theta-1}}{\beta x^\theta+(1-x)^\theta}$, {the first equality in} (\ref{eq:hazard}) may be rewritten as $h_{\nsG}(x)=h_F(x)T_{\beta,\theta}(F(x))$, {implying} immediate {shape} characterisations {for} some distributions {included in the oMO model}.


{\samepage
\begin{theorem}\
\label{thm:classes}
    \begin{enumerate}
    \item
    If $F\in\IHR$ and $\beta\leq 1$, then $G_{\beta,1}\in\IHR$.
    \item
    If $F\in\DHR$ and $\beta\geq 1$, then $G_{\beta,1}\in\DHR$.
    \item
    If $F\in\IOR$, then, for $\theta\geq 1$, $\nsG\in\IOR$.
    \end{enumerate}
\end{theorem}}
\begin{proof}
The result is immediate once we verify the monotonicity properties of $T_{\beta,\theta}$. {It is} easily {verified} that $T_{\beta,\theta}$ is monotone only for $\theta=1$, that $T_{\beta,1}$ is increasing for $\beta\leq1$, and that $T_{\beta,1}$ is decreasing for $\beta\geq 1$. With regard to the preservation of the IOR property, it follows directly
from (\ref{eq:def1}) taking into account that $\theta \geq 1$.
\end{proof}

{Note that, as for Proposition~\ref{prop:conv}, when $\theta\ne1$, no conclusion can be drawn about the monotonicity of the hazard of $\nsG$, as the corresponding $T_{\beta,\theta}$ is not monotonous. However, shifting our interest for the odds rate, part 3. in Theorem~\ref{thm:classes} provides {results for $\theta\geq 1$}. {Therefore, the odds rate shows greater flexibility in capturing the varying behaviours of model (\ref{eq:def1}).}
}

Given that $\nsG$ is a transformation of $F$, it is natural to compare the baseline 
with the transformed distribution. Specifically, we are interested in understanding how the transformation applied to $F$ affects key reliability properties and relationships.

\begin{theorem}\
\label{thm:lr}
    \begin{enumerate}
    \item
    If $\beta\leq 1$, then $F\leqlr G_{\beta,1}$.
    \item
    If $\beta\geq 1$, then $F\geqlr G_{\beta,1}$.
    \end{enumerate}
\end{theorem}
\begin{proof}
It is easily verified that $\frac{g_{\beta,1}(x)}{f(x)} = \frac{\beta}{(1+(\beta-1)F(x))^2}$, so the result follows immediately.
\end{proof}


\begin{corollary}\
\label{cor:hr+}
    \begin{enumerate}
    \item
    If $\beta\leq 1$, then $F\leqrhr G_{\beta,1}$, $F\leqhr G_{\beta,1}$ and $F\leqst G_{\beta,1}$.
    \item
    If $\beta\geq 1$, then $F\geqrhr G_{\beta,1}$, $F\geqhr G_{\beta,1}$ and $F\geqst G_{\beta,1}$.
    \end{enumerate}
\end{corollary}
\begin{proof}
{This is an immediate consequence of Theorem~\ref{thm:lr} and Theorem~1.C.1 in \cite{shaked2007}.}
\end{proof}

Note that when $\theta=1$, the following explicit bounds for $h_{G_{\beta,1}}$ are immediate:
\begin{enumerate}
\item
For $\beta\leq 1$, $\beta h_F(x) \leq h_{G_{\beta,1}}(x) \leq h_F(x)$,
\item
For $\beta\geq 1$, $h_F(x) \leq h_{G_{\beta,1}}(x) \leq \beta h_F(x)$.
\end{enumerate}

The previous results mention comparability for $G_{\beta,1}$. This choice for $\theta$ is the only one allowing for comparability results, as stated next.
\begin{corollary}
\label{cor:noncomp}
For $\theta\ne 1$, $F$ and $\nsG$ are not comparable with respect to the usual stochastic order. Therefore, they are also not comparable with respect to $\leqhr$, $\leqrhr$ or $\leqlr$ stochastic orders.
\end{corollary}
\begin{proof}
We need to look at
$$
\nsGbar(x)-\Fbar(x) \eqs \frac{\Fbar(x)^{\theta-1}}{\beta F(x)^{\theta}+\Fbar(x)^{\theta}}-1,
$$
so, the conclusion follows by analysing the sign of {$S_{\beta,\theta}(x) = \frac{(1-x)^{\theta-1}}{\beta x^\theta +(1-x)^{\theta}}-1$, for $x\in [0,1]$.}
Differentiating, one finds $S_{\beta,\theta}^\prime(x)\eqs (1-x)^\theta-\beta x^{\theta-1}(\theta-x)$. For $\theta\ne1$, one has $S_{\beta,\theta}(0)=0$, $S_{\beta,\theta}^\prime$ is positive for $x$ near 0 if $\theta>1$, and is negative if $\theta<1$. Finally, noting that $S_{\beta,\theta}(1)=-1$ if $\theta>1$, and $S_{\beta,\theta}(1)=+\infty$ if $\theta<1$, the first part of the result is proved. The second part is a consequence of Theorem~1.C.1 in \cite{shaked2007}
\end{proof}

{
\begin{example}
\label{ex:Prentice}
As an example of the usefulness of model (\ref{eq:def1}), we consider the data of Veteran's Administration lung cancer trial reported by \cite{Prentice1973}, that was analysed using a proportional odds model by \cite{Bennett1983}. The data describes the survival days of the 97 patients that had no prior treatment and two covariates, a performance status and tumor type. The analysis separated patients in two groups: low performance status (less or equal to 50) and high performance status. As noted in \cite{Bennett1983}, the log-odds of the two groups have, approximately, a constant difference, suggesting an affine relation, which served as a justification to apply a methodology based on a proportional odds model. However, running a linearity approximation test between the two sets of odds or log-odds clearly indicates that a linear relation between the odds can only explain about 75\% of the variability observed in the data, while assuming the linearity of the log-odds, one can explain about 90\% of the variability. Moreover, the estimates of the linear coefficients clearly suggests using model (\ref{eq:def1}) with $\beta=4.4324$ and $\theta=0.6822$. Besides, the group with high performance status has 72 observations, while the low performance status only contains 25 points. Therefore, to estimate the distribution, it is convenient to take as baseline distribution the one describing the high performance status group and then use the parameters mentioned above to get an estimate for the distribution of the survival days {for the low performance status group}. For the high performance status an estimate suggests a $\Gamma(1.13,116.6)$. Hence, the density function for the survival of the low performance status group is approximated by
$$
g_{4.4324,0.6822}(x)=4.4324\times0.6822\frac{f(x)F^{-0.3178}(x)\Fbar^{-0.3178}(x)}{\left(4.4324 F^{0.6822}(x)+\overline{F}^{0.6822}(x)\right)^2}
$$
where $f(x)$ and $F(x)$ are, respectively, the density and distribution functions of the $\Gamma(1.13,116.6)$ distribution. {Taking into account Theorem~\ref{thm:classes}, although the particular $F$ is \IOR, we cannot derive the same property about $G_{4.4324,0.6822}$.}
\end{example}}

{Note that, by construction, the oMO model leads to ordering or shape properties only for $\theta =1$. Although the underlying motivations were of a different nature, the model discussed in the next section allows for results with $\theta\ne1$.}

\section{The distorted odds-Marshall-Olkin model}
\label{sec:distmodel}
The previous section studied stochastic ordering relations between distributions defined by a specific transformation of the odds function, having in mind the possibility of mixing the proportionality of the odds ratio and of the log-odds ratio, each controlled by an appropriate parameter.
Observe that the odds function of the $\nsG$ distribution appears as a distortion of the odds ratio $\Lambda_F$ of the baseline distribution function $F$, adding a layer of flexibility in shaping distributional properties.
The model introduced in Definition~\ref{def:model1} {mimics the construction of the PHR model}, where a survival function is {transformed into} $\overline{F}^\theta(x)$ for some $\theta>0$. However, the PHR model is not covered {by the family of distributions} introduced in  Definition~\ref{def:model1}.
In fact, the odds function {for} the PHR model is of the form $(1+\Lambda_F(x))^\theta-1$.
It is worth noting that this latter form corresponds to transforming the underlying distribution $F$ by the quantile function of a Pareto distribution with survival function $(x+1)^{-\frac1\theta}$. This observation will be explored later in Section~\ref{sec:GLL} in more generality. Nevertheless, the general form of the odds function for the PHR model suggests an extension of the transformation used in Definition~\ref{def:model1}, targeted at fine tuning the tail behaviour.

\begin{definition}
\label{def:model2}
Let $\alpha \geq 0$, $\beta,\,\theta>0$. Given a baseline distribution function $F$, we define the distorted odds-Marshall-Olkin (d-oMO) distribution functions $\nG$ by
\begin{equation}
\label{eq:def2}
\Lambda_{\nG}(x)=\beta\left((\alpha+\Lambda_F(x))^\theta-\alpha^\theta\right).
\end{equation}
\end{definition}

It is obvious that the model introduced in Definition~\ref{def:model1} is a particular case of (\ref{eq:def2}), taking
$\alpha=0$, while the PHR model is obtained by choosing $(\alpha,\beta,\theta)=(1,1,\theta)$.
{Moreover, note that this extended model unifies the proportionality models we have discussed (odds, log-odds and hazard rate), offering seamless transition between them.}

Taking into account that $\Lambda_{\nG}(x)=\frac1{\nGbar(x)}-1$, the following explicit representation for the distributions $\nG$ introduced in Definition~\ref{def:model2} is immediate:
\begin{equation}
\label{eq:nGbar12}
\nGbar(x) = \frac{1}{1+\beta((\alpha+\Lambda_F(x))^\theta - \alpha^\theta)},
\end{equation}
while the density function is represented as
$$
g_{\alpha,\beta,\theta}(x) = \frac{\beta\theta(\alpha+\Lambda_F(x))^\theta}{\Fbar(x)(\alpha\Fbar(x)+F(x))}\Gbar_{\alpha,\beta,\theta}^2(x) f(x)
= \beta \theta(\alpha+\Lambda_F(x))^{\theta-1}\frac{\Gbar_{\alpha,\beta,\theta}^2(x)}{\Fbar(x)} h_{F}(x).
$$

\begin{remark}
The distribution function $\nG$ can be written as
\begin{equation}
\label{eq:nGbar12alt}
\nG(x) = \frac{\beta\left((\alpha+(1-\alpha)F(x))^\theta - (\alpha\Fbar(x))^\theta\right)}{\Fbar^\theta(x)+\beta\left((\alpha+(1-\alpha)F(x))^\theta -(\alpha \Fbar(x))^\theta\right)}.
\end{equation}
In the special case where $\alpha=1$, $G_{1,\beta,\theta}$ is the recently defined MPHR model introduced in \cite{Bala2018}. \cite{DasKayal2021} later extended this model by incorporating a scale parameter, calling it MPHRS. Similarly, our models can be generalised by introducing a scale parameter in the same way.
\end{remark}

\begin{remark}
\label{rem:param}
The family of distributions $\nG$ depends on three parameters. Here is a brief description of the effect of each one of them.
{All the parameters, as they increase, shift the mass towards the origin, introducing skewness and lighter tails. The parameter $\beta$ has a small effect on the mass shifting, affecting mainly mass concentration. The parameter $\theta$ has a dramatic effect in both shifting the mass closer to the origin and the concentration (take into consideration the vertical scale of the plots), hence contributing very significantly to skewness and lighter tails.}
An illustration of these effects can be found in Figure~\ref{fig3-5}.
\begin{figure}
\centering
\includegraphics[scale=.27]{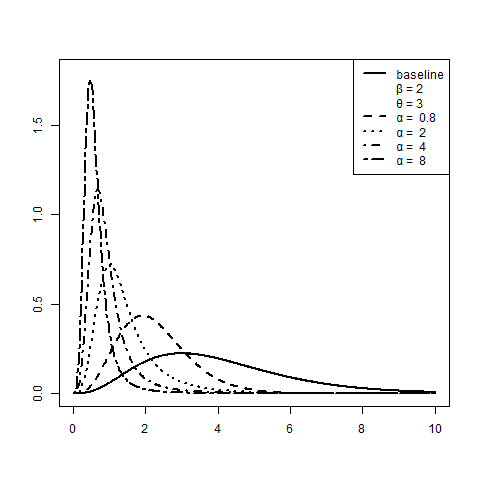}\includegraphics[scale=.27]{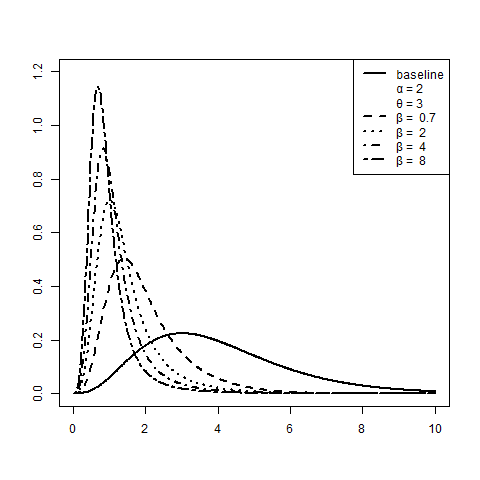}\includegraphics[scale=.27]{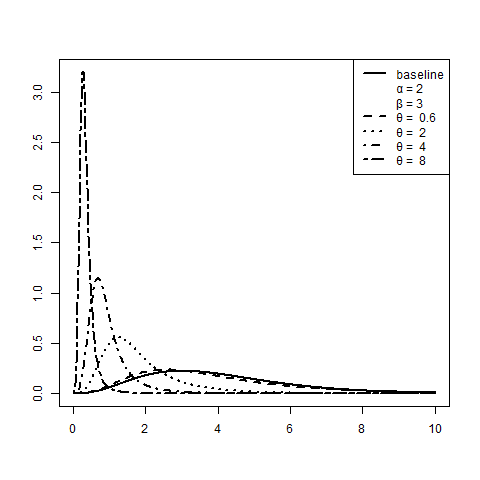}
\caption{Densities of the baseline distribution ($\Gamma(4,1)$) and $\nG$.}
\label{fig3-5}
\end{figure}
\end{remark}

Although it seems that the $\nG$ family is not closed under formation of maximums, that is, in general the distribution of the form $\nG^n(x)$, $n\geq 2$, {is not included in the d-oMO model}, we may still find an \textit{extreme geometrical stability property} (see \cite{marshall1997}).
\begin{theorem}
\label{thm:geomstab}
Let $X_1,X_2,\ldots$ be independent and with distribution function $\nG$, for some fixed values of $\alpha\geq0$, $\beta,\theta>0$, and consider $N$, independent from the $X_n$, with geometric distribution, $P(N=n)=p(1-p)^{n-1}$, $n\geq 1$, for some $p\in[0,1]$. Define $U=\min\{X_1,\ldots,X_N\}$ and $V=\max\{X_1,\ldots,X_N\}$. Then, the distribution function of $U$ and $V$ are $G_{\alpha,\frac{\beta}{p},\theta}$ and $G_{\alpha,{\beta}{p},\theta}$, respectively. Or, equivalently, the family of distributions $\nG$ has geometric extreme stability.
\end{theorem}
\begin{proof}
Proceeding by conditioning, the distribution function of $U$ is
$$
\Fbar_U(x)=\sum_{n=1}^\infty \nGbar^n(x)p(1-p)^{n-1}=\frac{p\nGbar(x)}{1-(1-p)\nGbar(x)}.
$$
Using now the representation for $\nGbar$ that follows from (\ref{eq:nGbar12alt}),
the result is immediate. The case of $V$ is treated analogously.
\end{proof}

\subsection{Preservation of monotonicity properties}
Given the expressions above, we have the following representation for the hazard rate function {for the distributions in the d-oMO model}:
$$
h_{\nG}(x) = \frac{g_{\alpha,\beta,\theta}(x)}{\nGbar(x)}
= \frac{\beta \theta(\alpha+\Lambda_F(x))^{\theta-1}}{1+\beta\left (\left(\alpha+\Lambda_F(x)\right)^\theta-\alpha^\theta\right)}\cdot\frac{h_{F}(x)}{\Fbar(x)} = \beta \theta h_{F}(x) T_{\alpha,\beta,\theta}(\Lambda_F(x)),
$$
where
\begin{equation}
\label{eq:W}
T_{\alpha,\beta,\theta}(x) = \frac{\left (\alpha+x\right )^{\theta-1}(x+1)}{1+\beta\left (\left(\alpha+x\right)^\theta-\alpha^\theta\right )}, \quad {x\geq 0}.
\end{equation}
Hence, we may prove monotonicity properties for $h_{\nG}$ looking at the monotonicity of $T_{\alpha,\beta,\theta}$, which will be addressed via $U_{\alpha,\beta,\theta}(x)=\frac{1}{T_{\alpha,\beta,\theta}(x)}$, for simplicity.
After differentiation and some simple  algebraic manipulation, one gets
{$U_{\alpha,\beta,\theta}^\prime(x)=\frac{D_{\alpha,\beta,\theta}(x)}{(x+1)^2(\alpha+x)^{\theta}}$,}
where $D_{\alpha,\beta,\theta}(x)= \beta(1-\alpha)(\alpha+x)^\theta+(\beta \alpha^\theta-1)(\theta x+\alpha+\theta-1)$, and the sign of $U_{\alpha,\beta,\theta}^\prime$ coincides with the sign of $D_{\alpha,\beta,\theta}$. We have that $D_{\alpha,\beta,\theta}^\prime(x)=\theta \beta(1-\alpha)(\alpha+x)^{\theta-1}+\theta (\beta \alpha^\theta-1)$ and $D_{\alpha,\beta,\theta}^{\prime\prime}(x)=(1-\alpha)\beta\theta(\theta-1)(\alpha+x)^{\theta-2}$. Therefore, $D_{\alpha,\beta,\theta}^{\prime\prime}(x)\eqs {\rm sgn}((1-\alpha)(\theta-1))$, so $D_{\alpha,\beta,\theta}^\prime$ is either increasing or decreasing. Now, the sign of $D_{\alpha,\beta,\theta}(0)=\alpha^\theta\beta\theta-\alpha-(\theta-1)$ will play a significant role.
\begin{theorem}
\label{thm:GLL-IHR}
Let $\nG$ be given by (\ref{eq:def2}) and $D_{\alpha,\beta,\theta}$ be the polynomial defined above.
    \begin{enumerate}
    \item
    If $D_{\alpha,\beta,\theta}(0)<0$, $(1-\alpha)(\theta-1)<0$, and $F\in\IHR$, then $\nG\in\IHR$.
    \item
    If $D_{\alpha,\beta,\theta}(0)>0$, $(1-\alpha)(\theta-1)>0$, and $F\in\DHR$, then $\nG\in\DHR$.
    \item
    If $\alpha=1$ or $\theta=1$, 
    {$\beta\leq 1$} and $F\in\IHR$, then $\nG\in\IHR$.
    \item
    If $\alpha=1$ or $\theta=1$, 
    {$\beta\geq1$} and $F\in\DHR$, then $\nG\in\DHR$.
    \end{enumerate}
\end{theorem}
\begin{proof}
In the first case, $D(x)<0$ for every $x>0$. Hence $U_{\alpha,\beta,\theta}^\prime$ is always negative, so $U_{\alpha,\beta,\theta}$ is decreasing and, therefore, $T_{\alpha,\beta,\theta}(x)=\frac1{U_{\alpha,\beta,\theta}(x)}$ is increasing, so the conclusion is straightforward. The remaining cases are analogous, reversing signs and monotonicity directions for cases 2. and 4.
\end{proof}

Note that this result extends Theorem~\ref{thm:classes}, allowing now for an interplay of the different parameters. 
{Moreover, it is the presence of the parameter $\alpha$ that allows concluding about the monotonicity of the hazard rate for $\theta\ne1$, which was out of reach in Theorem~\ref{thm:classes}.}

The preservation of the monotonicity of the odds ratio is easily described in analogous terms, extending the final part of Theorem~\ref{thm:classes}.

\begin{theorem}\
\label{thm:GLL-IOR}
    \begin{enumerate}
    \item
    If $\theta\geq1$ and $F\in\IOR$, then $\nG\in\IOR$.
    \item
    If $\theta\leq1$ and $F\in\DOR$, then $\nG\in\DOR$.
    \end{enumerate}
\end{theorem}
\begin{proof}
Just note that $\lambda_{\nG}^\prime(x) \eqs \lambda_F^\prime(x)(\alpha+\Lambda_F(x))+(\theta-1)\lambda_F^2(x)$, and the conclusion is immediate.
\end{proof}
	
\subsection{Stochastic comparisons between $\nG$ and $F$}
We now address some stochastic ordering relations between the baseline distribution $F$ and the family of transformed distributions $\nG$ {in the d-oMO model}.
\begin{theorem}\
\label{thm:GLLst}
    \begin{enumerate}
    \item
    If $\theta>1$ and $\alpha^{\theta-1}\beta\theta>1$, then $\nG\leqst F$.
    \item
    If $\theta<1$ and $\alpha^{\theta-1}\beta\theta<1$, then $\nG\geqst F$.
    \end{enumerate}
\end{theorem}
\begin{proof}
We need to characterise the sign of
$$
\nGbar(x)-\Fbar(x) 
=  \frac{1}{1+\beta((\alpha+\Lambda_F(x))^\theta - \alpha^\theta)} - \frac{1}{\Lambda_F(x) +1} = H(\Lambda_F(x)),
$$
where $H(x) = \frac{1}{1+\beta((\alpha+x)^\theta - \alpha^\theta)} - \frac{1}{x +1}$, which has the same sign variation as $P(x) = x-\beta(\alpha+x)^{\theta} +\alpha^\theta \beta$.
After differentiation, we have $P^\prime(x) = 1-\theta\beta(\alpha+x)^{\theta-1}$ and $P^{\prime\prime}(x) = -\beta\theta(\theta-1)(\alpha +x)^{\theta-1}$. When $\theta>1$, $P^{\prime\prime}(x)<0$, so $P^\prime(x)$ is decreasing. If $P^\prime(0) = 1-\beta\theta \alpha^{\theta-1} < 0$ it follows that $P^\prime(x)<0$, for every $x>0$, hence $P(x)$ is decreasing. Since  that $P(0) = 0$ we have the negativeness of $P(x)$. The case $\theta<1$ is handled analogously.
\end{proof}

Sufficient conditions for the hazard rate order follow immediately by remarking that
$$
H^\ast(x)=\frac{h_{\nG}(x)}{h_F(x)} = \beta\theta T_{\alpha,\beta,\theta}(\Lambda_F(x)),
$$
where $T_{\alpha,\beta.\theta}$ is defined by (\ref{eq:W}). Noting that $H^\ast(0)=\alpha^{\theta-1}$, taking into account the properties of $T_{\alpha,\beta.\theta}$ mentioned above, the following statement is obvious.
\begin{theorem}\
    \begin{enumerate}
    \item
    If $\alpha^{\theta-1}>1$ and $T_{\alpha,\beta,\theta}$ is increasing, then $\nG\leqhr F$.
    \item
    If $\alpha^{\theta-1}<1$ and $T_{\alpha,\beta,\theta}$ is decreasing, then $\nG\geqhr F$.
    \end{enumerate}
\end{theorem}

For a characterisation of the monotonicity of $T_{\alpha,\beta,\theta}$, please see the discussion {preceding} Theorem~\ref{thm:GLL-IHR}.

Now, the likelihood order follows easily.

\begin{theorem}\
    \begin{enumerate}
    \item
    Assume that $\theta>1$ and $F\leqhr\nG$. Then $\nG{\leqlr} F$.
    \item
    Assume that $\theta<1$ and $F\geqhr\nG$. Then $\nG{\geqlr} F$.
    \end{enumerate}
\end{theorem}
\begin{proof}
Note that
$$
\frac{g_{\alpha,\beta,\theta}(x)}{f(x)} = \frac{\beta\theta(\alpha+\Lambda_F(x))^{\theta} \Gbar^2(x)}{\Fbar(x)(\alpha\Fbar(x)+F(x))} = \beta\theta(\alpha+\Lambda_F(x))^{\theta-1}\left(\frac{\Gbar(x)}{\Fbar(x)}\right)^2.
$$
The monotonicity of the first parenthesis of the final expression on the right is fully defined by the sign of the exponent, while the monotonicity of the second term depends on monotonicity the hazard rate order between the distributions $F$ and $\nG$.
\end{proof}

\section{An enlarged log-logistic family of distributions}
\label{sec:GLL}

We have treated the {oMO and d-oMO} models, introduced in Definitions~\ref{def:model1} and \ref{def:model2} by defining distributions through their odds functions, {} as distortions of some given underlying odds function $\Lambda_F$. Naturally, we may instead consider the new odds function as a distortion of the initial distribution function $F$. We mentioned, just before Definition~\ref{def:model2}, that the odds function for the PHR model corresponds to transforming $F$ by {the} quantile function {of a Pareto distribution}. This approach may be extended to the full class of models considered in Definition~\ref{def:model2}, leading to the introduction of a new family of distributions.

\begin{definition}
\label{def:GLL}
{The} enlarged log-logistic distribution with parameters $\alpha\geq0$, $\beta,\,\theta>0$, denoted with $\GLL(\alpha,\beta,\theta)$ has distribution function
\begin{equation}
\label{eq:GLL}
\cdfgll(x)=1-\frac{1}{\left(\frac{x}{\beta}+\alpha^\theta\right)^{\frac{1}{\theta}}+1-\alpha}, \quad x\geq 0.
\end{equation}
\end{definition}

The parameters $\beta$ and $\frac1\theta$ are obviously scale and shape parameters, respectively, and $\alpha$ is a second shape parameter, having an effect on the asymmetry, skewness, and tail weight of the distribution.
Moreover, it is straightforward to verify that $K_{0,1,1}$ is the standard log-logistic, already introduced before and denoted with $\mathcal{L}$, while $K_{0,\beta,\theta}$ represents the log-logistic with distribution function $\mathcal{L}_{\beta,\frac1\theta}(x)=1-\left((\frac{x}{\beta})^{\frac1\theta}+1\right)^{-1}$. {Moreover, the distributions $K_{1,\beta,\theta}$ correspond to the Pareto family.}

Explicit expressions for the density $k_{\alpha,\beta,\theta}$, hazard rate $h_{\alpha,\beta,\theta}$, and quantile function $\cdfgll^{-1}$ for the distribution function $\cdfgll$ are given below:
$$
k_{\alpha,\beta,\theta}(x)
=\frac{\left(\frac{x}{\beta}+\alpha^\theta\right)^{\frac{1}{\theta}-1}}
{\beta\theta\left(\left(\frac{x}{\beta}+\alpha^\theta\right)^{\frac{1}{\theta}}+1-\alpha\right)^2},
\qquad
h_{\alpha,\beta,\theta}(x) =
\frac{\left(\frac{x}{\beta}+\alpha^\theta\right)^{\frac{1}{\theta}-1}}
{\beta\theta\left(\left(\frac{x}{\beta}+\alpha^\theta\right)^{\frac{1}{\theta}}+1-\alpha\right)}, \quad x\geq 0.
$$
and
$$
\cdfgll^{-1}(u)=\beta\left(\left(\frac{1}{1-u}+\alpha-1\right)^\theta-\alpha^\theta\right),\quad u\in[0,1].
$$
{An illustration of the effect of the parameters is shown in Figure~\ref{fig6-8}, describing the behaviour of the density function. It is clear that mass concentrates near the origin, with a shift to the right for $\alpha<1$. On the other hand, increasing $\theta$ produces heavier tails.}
\begin{figure}
\centering
\includegraphics[scale=.27]{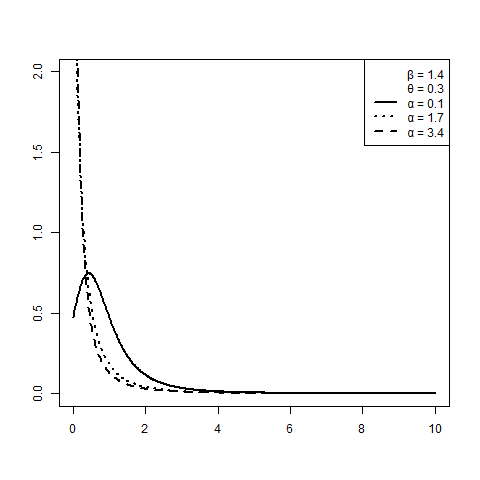}\includegraphics[scale=.27]{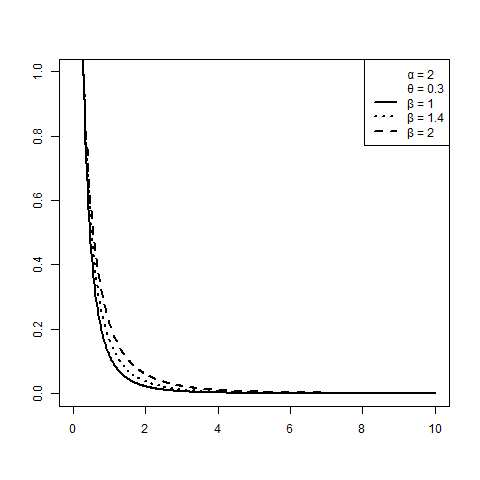}\includegraphics[scale=.27]{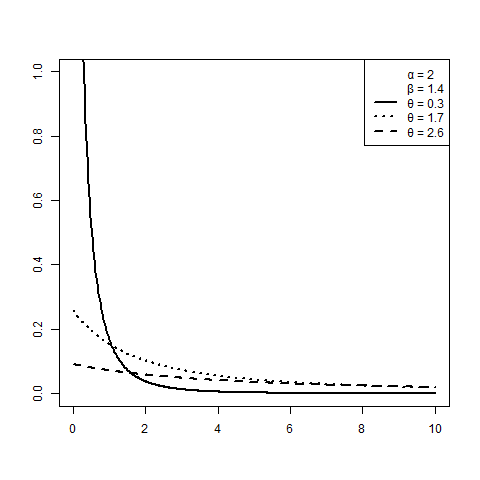}
\caption{Some plots for the density of $\cdfgll$.}
\label{fig6-8}
\end{figure}

{It is now obvious} that $\Lambda_{\nG}(x)=\cdfgll^{-1}\circ F(x)$. Therefore, ordering properties within the $\cdfgll$ family translate easily {into the} distributions {in the d-oMO model}, the convexity of the odds of $\nG$ being the most obvious, corresponding to the convex transform order between $\cdfgll$ and $F$.
In other words, the convexity properties of the baseline distribution $F$ with respect to $\cdfgll$ are inherited by $\nG$. Recall that the convexity of the odds {function} defines interesting classes, namely the IOR and DOR families of distributions (see \cite{oddspaper}). This naturally leads to an interest in exploring stochastic ordering relationships within the family of distributions defined by (\ref{eq:GLL}).

The increasingness of the hazard rate or the odds rate is simple to characterise, as described next.


{\samepage
\begin{theorem}\
\label{thm:GLL-IHRIOR}
    \begin{enumerate}
    \item
    If $\alpha+\theta>1$ then $\cdfgll\in\DHR$, for every $\beta>0$.
    \item
    If $\theta\leq1$ then $\cdfgll\in\IOR$, while if $\theta\geq1$ then $\cdfgll\in\DOR$.
    \end{enumerate}
\end{theorem}}
\begin{proof}
For part 1., note that
$$
h_{\alpha,\beta,\theta}^\prime(x)=
-  \left(\theta\left(\frac{x}{\beta} +\alpha^\theta\right)^{\frac{1}{\theta}} +(1-\alpha)(\theta-1)\right).
$$
Since $h_{\alpha,\beta,\theta}^\prime(0)=-(\alpha+\theta-1)$, it follows that $h_{\alpha,\beta,\theta}^\prime(x)<0$ for every $x\geq 0$, given that $h^\prime$ is decreasing. For part 2., the odds rate of $\cdfgll$ is given by $\lambda_{\cdfgll}(x)=\frac{1}{\beta\theta}\left(\frac{x}{\beta}+\alpha^\theta\right)^{\frac1\theta-1}$, so the conclusion is obvious.
\end{proof}

The following result characterises the $\leqst$-order comparability within the \GLL\ family.
\begin{theorem}
\label{thm:genGLLst}
Assume the parameters $\alpha\geq0$, $\beta,\,\theta>0$ and $\alpha_1\geq0$, $\beta_1,\,\theta_1>0$ of the enlarged log-logistic distribution functions (\ref{eq:GLL}) satisfy one of the following assumptions:
    \begin{description}
    \item[(ST1)]
    $\theta<\theta_1$, 
                 {$\alpha^{\theta-1}\beta\theta < \alpha_1^{\theta_1-1}\beta_1\theta_1$}
    and 
    {$\alpha_1(1-\theta)-\alpha(1-\theta_1)\geq0$};
    \item[(ST2)]
    $\theta=\theta_1$, $\beta<\beta_1$, 
                 {$\alpha^{\theta-1}\beta < \alpha_1^{\theta_1-1}\beta_1$}
    and $
    (1-\theta)\left(\alpha_1^\theta\beta_1-\alpha^\theta\beta\right)>0$.
    \end{description}
Then $\cdfgll\leqst\cdfgllalt$.
\end{theorem}
\begin{proof}
For the general set of parameters $(\alpha,\beta,\theta)$ denote $\overline{K}_{\alpha,\beta,\theta}(x)=1-\cdfgll(x)$, and define
$$
V(x)=\frac1{\overline{K}_{\alpha,\beta,\theta}(x)}-\frac1{\overline{K}_{\alpha_1,\beta_1,\theta_1}(x)}
=\left(\frac{x}\beta+\alpha^\theta\right)^{\frac1\theta} - \left(\frac{x}{\beta_1}+\alpha_1^{\theta_1}\right)^{\frac1{\theta_1}} + \alpha_1-\alpha.
$$
Noting that $\overline{K}_{\alpha_1,\beta_1,\theta_1}(x)-\overline{K}_{\alpha,\beta,\theta}(x)\eqs V(x)$, the proof is concluded if we prove that $V(x)\geq 0$, for every $x\geq0$. It is obvious that $V(0)=0$. We separate the two cases, according to which assumption is satisfied.
    \begin{description}
    \item[(ST1):]
    We have $V(+\infty)=\infty\times{\rm sgn}\left(\frac1\theta-\frac1{\theta_1}\right)=+\infty$. Differentiating, we find
    $$
    V^\prime(x)=\frac{1}{\beta\theta}\left(\frac{x}\beta+\alpha^\theta\right)^{\frac1\theta-1}
    -\frac{1}{\beta_1\theta_1}\left(\frac{x}{\beta_1}+\alpha_1^{\theta_1}\right)^{\frac1{\theta_1}-1},
    $$
    so $V^\prime(0)=\frac{\alpha^{1-\theta}}{\beta\theta}-\frac{\alpha_1^{1-\theta_1}}{\beta_1\theta_1}>0$. Now, if we prove that $V^\prime(x)\geq0$, for every $x\geq0$, it follows that $V$ is increasing, hence $V(x)\geq 0$, and the conclusion follows. Therefore, we need to prove that
    $$
    V^\prime(x)\geq0 \quad\Leftrightarrow\quad
    P(x)=\frac{\left(\frac{x}\beta+\alpha^\theta\right)^{\frac1\theta-1}}
              {\left(\frac{x}{\beta_1}+\alpha_1^{\theta_1}\right)^{\frac1{\theta_1}-1}}
    \geq \frac{\beta\theta}{\beta_1\theta_1}.
    $$
    Noting that $P(0)=\frac{\alpha^{1-\theta}}{\alpha_1^{1-\theta_1}}>\frac{\beta\theta}{\beta_1\theta_1}$ and $P(+\infty)=+\infty$, we now look at the monotonicity of $P$. Differentiating, one observes that $P^\prime(x)\eqs L(x)$, where $L(x)=\frac{1}{\beta\beta_1}\left(\frac1\theta-\frac1{\theta_1}\right)x
      +\frac{1-\theta}{\theta}\frac{\alpha_1^{\theta_1}}{\beta}
       -\frac{1-\theta_1}{\theta_1}\frac{\alpha^{\theta}}{\beta_1}$.
    The assumptions imply that both the slope and intercept of $L(x)$ are positive, hence $P$ is increasing, implying that $P(x)\geq\frac{\beta\theta}{\beta_1\theta_1}$, thus $V^\prime(x)=0$ has no solution.
    \item[(ST2):]
    This case is treated analogously, so we just highlight the relevant differences. We now have $V(+\infty)=\infty\times{\rm sgn}\left(\frac1\beta-\frac1{\beta_1}\right)=+\infty$, and $P^\prime(x)\eqs (1-\theta)\left(\alpha_1^\theta\beta_1-\alpha^\theta\beta\right)$, assumed to be positive.
    \end{description}
\end{proof}

The previous result allows for an immediate pointwise comparison of the odds ratio of the $\nG$ family.
\begin{corollary}
\label{cor:odds-point}
Let $\nG$ be given by (\ref{eq:def2}). Under either of the assumptions of Theorem~\ref{thm:genGLLst}, it holds that $\Lambda_{\nG}(x)\leq\Lambda_{\nGs}(x)$ for every $x\geq 0$.
\end{corollary}
\begin{proof}
Remember that $\Lambda_{\nG}(x)=\cdfgll^{-1}(F(x))$. Under the assumptions of Theorem~\ref{thm:genGLLst}, we have that $\cdfgll(x)\geq\cdfgllalt(x)$, for every $x\geq0$. But this is equivalent to $\cdfgll^{-1}(x)\leq\cdfgllalt^{-1}(x)$ for every $x\geq0$, so the result follows immediately.
\end{proof}

{Recalling the characterisation of the $\leqst$-order using the odds function (see Proposition~\ref{prop:oddsa}), the previous result can be rewritten as follows, describing conditions for the usual stochastic order within the distributions of the d-oMO model.
\begin{corollary}
\label{cor:odds-point-alt}
Let $\nG$ be given by (\ref{eq:def2}). Under either of the assumptions of Theorem~\ref{thm:genGLLst}, it holds that $\nGs\leqst\nG$.
\end{corollary}}

Conditions for the particular case of one parameter comparison, corresponding to the \GLL\ that characterise the PHR, the proportional odds or the proportional log-odds models, are immediate from Theorem~\ref{thm:genGLLst}. We state the result, for sake of completeness.

\begin{corollary}
\label{cor:genGLLst}
For the enlarged log-logistic distribution functions (\ref{eq:GLL}) we have that:
    \begin{enumerate}
    \item
    If $\alpha\geq\alpha_1\geq 0$ and $\theta\leq1$ then $\cdfgll\leqst K_{\alpha_1,\beta,\theta}$ for every $\beta>0$.
    \item
    If $\beta\leq\beta_1$ and $\theta\leq1$ then $\cdfgll\leqst K_{\alpha,\beta_1,\theta}$ for every $\alpha\geq0$.
    \item
    If $\theta<\theta_1\leq 1$, $\alpha^{\theta_1-\theta}>\frac\theta{\theta_1}$ and $\frac{1-\theta}{\alpha^\theta\theta}>\frac{1-\theta_1}{\alpha^{\theta_1}\theta_1}$ then
    $\cdfgll\leqst K_{\alpha,\beta,\theta_1}$ for every $\beta>0$.
    \end{enumerate}
\end{corollary}

We now prove a general set of conditions providing the $\leqhr$-comparability within the \GLL\ family.
\begin{theorem}
\label{thm:genGLLhr}
Assume the parameters $\alpha>0$, $\beta>0$, $\theta>0$ and $\alpha_1>0$, $\beta_1>0$, $\theta_1>0$ satisfy the following assumptions:
    \begin{description}
    \item[(HR1)] \textit{(i)} $\beta\theta\alpha^{\theta-1} \leq \beta_1\theta_1\alpha_1^{\theta_1-1}$, and \textit{(ii)} $\beta\theta\alpha^{\theta} \leq \beta_1\theta_1\alpha_1^{\theta_1}$,
    \item[(HR2)] $\theta < \theta_1$,
    \item[(HR3)] $(1-\alpha_1)(\theta_1-1) \geq 0$,
    \item[(HR4)] $\left(\frac1\alpha-1\right)(\theta-1) \leq \left(\frac1{\alpha_1}-1\right)(\theta_1-1)$.
    \end{description}
Then $\cdfgll \leqhr \cdfgllalt$.
\end{theorem}
\begin{proof}
We shall prove that
\begin{eqnarray*}
V(x) &=& \frac{1}{h_{\alpha_1,\beta_1,\theta_1}(x)}-\frac{1}{h_{\alpha,\beta,\theta}(x)} \\
 & = & \beta_1\theta_1\left(\frac{x}{\beta_1}+\alpha_1^{\theta_1}\right)
         +(1-\alpha_1)\beta_1\theta_1\left(\frac{x}{\beta_1}+\alpha_1^{\theta_1}\right)^{1-\frac1{\theta_1}} \\
  & &\qquad\quad - \beta\theta\left(\frac{x}{\beta}+\alpha^{\theta}\right)
         -(1-\alpha)\beta\theta\left(\frac{x}{\beta}+\alpha^{\theta}\right)^{1-\frac1{\theta}} \geq 0,
\end{eqnarray*}
which clearly implies the conclusion. We start by noting that, taking into account \textit{(HR1-i)} and \textit{(HR2)}, $V(0)=\beta_1\theta_1\alpha_1^{\theta_1-1}-\beta\theta\alpha^{\theta-1}\geq 0$ and $V(+\infty)=\infty\times{\rm sgn}(\theta_1-\theta)=+\infty$, hence the conclusion follows if we prove that $V$ is increasing. Direct differentiation gives
\begin{equation}
\label{eq:GLLhr-Vprime}
V^\prime(x)=\theta_1-\theta
    +(1-\alpha_1)(\theta_1-1)\left(\frac{x}{\beta_1}+\alpha_1^{\theta_1}\right)^{-\frac{1}{\theta_1}}
    -(1-\alpha)(\theta-1)\left(\frac{x}{\beta}+\alpha^{\theta}\right)^{-\frac{1}{\theta}},
\end{equation}
so, given \textit{(HR2)} and \textit{(HR3)}, the nonnegativity of $V^\prime$ follows if we prove that
$$
Q(x)=
     \frac{\beta_1^{\frac1{\theta_1}}}{\beta^{\frac1{\theta}}}
     \frac{\left(x+\beta\alpha^\theta\right)^{\frac1\theta}}
          {\left(x+\beta_1\alpha_1^{\theta_1}\right)^{\frac1{\theta_1}}}
\geq \frac{(1-\alpha)(\theta-1)}{(1-\alpha_1)(\theta_1-1)}.
$$
It is easily seen that \textit{(HR3)} and \textit{(HR1-ii)} imply that $Q^\prime(x)\geq0$ for every $x\geq 0$, hence $Q$ is increasing. Finally, \textit{(HR4)} means that $Q(0)\geq \frac{(1-\alpha)(\theta-1)}{(1-\alpha_1)(\theta_1-1)}$, so the theorem is proved.
\end{proof}

Theorem~\ref{thm:genGLLhr} does not allow to choose $\alpha=0$, therefore leaving out of the comparisons the important case of the log-logistic distribution $\mathcal{L}$, as the expression (\ref{eq:GLLhr-Vprime}) means, when taking $x=0$, that $\alpha$ appears as a denominator. The way out of this can be sorted adapting the expressions above by continuity when $\alpha\rightarrow 0$.
\begin{corollary}
\label{cor:genGLLhr-alpha0}
Assume the parameters $\beta>0$, $0<\theta\leq1$ and $\alpha_1>0$, $\beta_1>0$, $\theta_1>0$ satisfy (HR2) and (HR3). Then $K_{0,\beta,\theta}\leqhr\cdfgllalt$.
\end{corollary}
\begin{proof}
With respect to the proof of Theorem~\ref{thm:genGLLhr} note that, after allowing $\alpha\rightarrow 0$, we need that $\theta\leq1$ to fulfill the appropriate version of $Q(0)=0\geq\frac{\theta-1}{(1-\alpha_1)(\theta_1-1)}$.
\end{proof}

Moreover, note that Theorem~\ref{thm:genGLLhr} proof's argument depends crucially on $\theta<\theta_1$, and breaks down if we assume equality of these two parameters.
\begin{corollary}
\label{cor:genGLLhr-theta}
Assume the parameters $\alpha>0$, $\beta>0$, $\theta>0$ and $\alpha_1>0$, $\beta_1>0$ are such that $\alpha>\alpha_1$, $\beta\alpha^{\theta-1} \leq \beta_1\alpha_1^{\theta-1}$ and
$$
\left\{\begin{array}{l}
\mbox{if }\theta\geq1,\quad (1-\alpha)\beta^{\frac{1}{\theta}} \leq (1-\alpha_1)\beta_1^{\frac{1}{\theta}}, \\ \\
\mbox{if }\theta\leq1,\quad\beta\alpha^\theta<\beta_1\alpha_1^\theta.
\end{array}\right.
$$
Then $\cdfgll\leqhr K_{\alpha_1,\beta_1,\theta}$.
\end{corollary}
\begin{proof}
Rewrite the function
\begin{eqnarray*}
\lefteqn{V(x)=\frac{1}{h_{\alpha_1,\beta_1,\theta}(x)}-\frac{1}{h_{\alpha,\beta,\theta}(x)}} \\
 & = & \left(\beta_1\alpha_1^\theta-\beta\alpha^\theta\right)\theta
         +(1-\alpha_1)\beta_1\theta\left(\frac{x}{\beta_1}+\alpha_1^{\theta}\right)^{1-\frac1{\theta}}
         -(1-\alpha)\beta\theta\left(\frac{x}{\beta}+\alpha^{\theta}\right)^{1-\frac1{\theta}}.
\end{eqnarray*}
The assumptions imply that $V(0)\geq 0$ and $V(+\infty)\geq 0$, possibly equal to $+\infty$. The equation $V^\prime(x)=0$ translates into
$$
1+\frac{\beta\alpha^\theta-\beta_1\alpha_1^\theta}{x+\beta_1\alpha_1^\theta}
=\frac{\beta}{\beta_1}\left(\frac{1-\alpha}{1-\alpha_1}\right)^\theta,
$$
which may have at most one root for $x\geq 0$. Moreover, $V^\prime(0)=\frac{\alpha-\alpha_1}{\alpha\alpha_1}>0$, therefore $V$ starts increasing at $x=0$. Hence, $V(x)>0$ for every $x\geq0$, and the conclusion follows.
\end{proof}

Again, as for the general result, Corollary~\ref{cor:genGLLhr-theta} does not include the case $\alpha=0$, but this can be handled in exactly the same way as in Corollary~\ref{cor:genGLLhr-alpha0}. 
\begin{corollary}
\label{cor:genGLLhr-theta-alpha0}
Assume the parameters $\beta>0$, $\theta>0$ and $\alpha_1>0$, $\beta_1>0$ are such that $\beta^{\frac1\theta}\leq(1-\alpha_1)\beta_1^{\frac1\theta}$. Assume, further, than one of the following conditions is satisfied:
    \begin{description}
    \item[(HR5)] $(1-\alpha_1)(\theta-1)\geq 0$;
    \item[(HR6)] $(1-\alpha_1)(\theta-1)<0$ and
      $\alpha_1+(1-\alpha_1)\beta_1^{1-\frac1\theta}\frac{(1-\alpha_1^\theta)^{1-\frac1\theta}-\beta}
                                                         {\left(\beta_1(1-\alpha_1^\theta)-\beta\right)^{1-\frac1\theta}}\geq0$.
    \end{description}
Then $K_{0,\beta,\theta}\leqhr K_{\alpha_1,\beta_1,\theta}$.
\end{corollary}
\begin{proof}
We need to look now at the sign of
$$
V(x)=\frac{1}{h_{\alpha_1,\beta_1,\theta}(x)}-\frac{1}{h_{0,\beta,\theta}(x)}
    =\beta_1\alpha_1^\theta\theta
         +(1-\alpha_1)\beta_1\theta\left(\frac{x}{\beta_1}+\alpha_1^{\theta}\right)^{1-\frac1{\theta}}
         -\beta\theta\left(\frac{x}{\beta}\right)^{1-\frac1{\theta}}.
$$
We have $V(0)=\beta_1\alpha_1^{\theta-1}\theta>0$. Moreover,
$$
V(+\infty)=\left\{\begin{array}{lcl}
\displaystyle
\infty\times{\rm sgn}\left((1-\alpha_1)\beta_1^{\frac1\theta}-\beta^{\frac1\theta}\right)
 & \quad & \displaystyle\mbox{if }1-\frac1\theta>0, \\
 & & \\
\displaystyle\beta_1\alpha_1^\theta & \quad & \displaystyle\mbox{if }1-\frac1\theta<0.
\end{array}\right.
$$
Therefore, under our assumptions, $V(+\infty)=+\infty$ for every $\theta>0$. Seeking for extreme points of $V$, we need to solve $V^\prime(x)=0$, which translates into
$$
P(x)=\frac{x}{x+\beta_1\alpha_1^\theta}
=\frac{\beta}{\beta_1}\frac{1}{(1-\alpha_1)^\theta}.
$$
It is easy to verify that $P$ is increasing, $P(0)=0$, $P(+\infty)=1$, and the right hand side of he equation is less or equal than 1, so this equation has exactly one solution, equal to $x_0=\frac{\beta\beta_a\alpha_1^\theta}{\beta_1(1-\alpha_1^\theta)-\beta}$. Assuming \textit{(HR5)}, it follows that $V^\prime(x)\geq 0$, for every $x\geq 0$, hence $V$ remains positive. If assuming \textit{(HR6)}, $V$ has a minimum at $x_0$, and our assumptions mean that $V(x_0)\geq 0$ so, again, we conclude that $V$ stays positive, thus concluding the proof.
\end{proof}

Finally, a characterisation of convex transform order relationships.
\begin{theorem}
\label{thm:GLLclasses}
For the enlarged log-logistic distribution functions (\ref{eq:GLL}) we have that:
    \begin{enumerate}
    \item
    If $\theta\leq\theta_1$ and $\alpha(\theta_1-1)+\alpha_1(1-\theta)\geq0$, then for every $\beta,\,\beta_1>0$, $\cdfgll\leqc\cdfgllalt$.
    \item
    If $\theta\geq\theta_1$ and $\alpha(\theta_1-1)+\alpha_1(1-\theta)\leq0$, then for every $\beta,\,\beta_1>0$, $\cdfgllalt\leqc\cdfgll$.
	\end{enumerate}
\end{theorem}
\begin{proof}
First note that as the $\beta$ is a scale parameter and the convex transform order is invariant with respect to scale parameters, we may assume that $\beta=\beta_1=1$. We need to look at the convexity/concavity of
$$
\psi(x)=K_{\alpha_1,1,\theta_1}^{-1}\circ K_{\alpha,1,\theta}(x)
 =\left(\left(x+\alpha^\theta\right)^{\frac1\theta}+\alpha_1-\alpha\right)^{\theta_1}-\alpha_1^{\theta_1}.
$$
Simple differentiation and simplification show that $\psi^{\prime\prime}(x)\eqs(\theta_1-\theta)\left(x+\alpha^\theta\right)^{\frac1\theta}+(1-\theta)(\alpha_1-\alpha)$. Therefore, $\psi$ is convex if $\theta_1-\theta\geq0$ and $\psi^{\prime\prime}(0)=\alpha(\theta_1-1)+\alpha_1(1-\theta)>0$, and it is concave if both these two inequalities are reversed.
\end{proof}

The following particular cases are now obvious.
\begin{corollary}
\label{cor:GLLclasses}
For the enlarged log-logistic distribution functions (\ref{eq:GLL}) we have that:
    \begin{enumerate}
    \item
    If $\theta \geq 1$, then for every $\alpha\geq 0$, $\mathcal{L}=K_{0,\beta,1}\leqc\cdfgll\leqc K_{0,\beta,\theta}$.
    \item
    If $\theta\leq 1$, then for every $\alpha \geq 0$, $K_{0,\beta,\theta}\leqc\cdfgll\leqc K_{0,\beta,1}=\mathcal{L}$.
	\end{enumerate}
\end{corollary}

\begin{remark}\label{inclusion}
As mentioned above, the $\IOR$ family may be characterised as the class of distributions that are dominated, with respect to the convex transform order, by the standard log-logistic $K_{0,1,1}$ (which is equivalent, for this purpose, to $K_{0,\beta,1}$, for every $\beta>0$). Denote with $D_{\alpha,\beta,\theta}$ the family of distributions that are dominated, with respect to the convex transform order, by the $\cdfgll$ distribution. Then, we have that $\IOR=D_{0,\beta,1}$, for every $\beta>0$. Moreover, the transitivity of the $\leqc$-ordering implies that, for $\theta\geq1$ and $\alpha\geq0$, $\IOR=D_{0,\beta,1}\subset D_{\alpha,\beta,\theta}\subset D_{0,\beta,\theta}$. This inclusion implies that, for this choice of parameters, the IOR class remains nested within this more general family, hence meaning that the requirement that $G\in D_{\alpha,\beta,\theta}$ is less stringent that $G\in\IOR$. As emphasized in \cite{oddspaper}, the $\IOR$ already encompasses several well-known distributions with interesting shape properties, namely, allows heavy tailed distributions or for bathtub shaped hazard rates.
\end{remark}

In Theorem~\ref{thm:GLL-IOR}, we described conditions implying the monotonicity of the odds rate $\lambda_{\nG}$. This monotonicity, following the \cite{oddspaper}, translates into either $\nG\leqc\mathcal{L}=K_{0,1,1}$, equivalent to $\nG\in\IOR$, or $\mathcal{L}=K_{0,1,1}\leqc\nG$, equivalent to $\nG\in\DOR$. We may now describe a more general form of the convex transform relations between the $\nG$ and $\cdfgll$ families of distributions.
\begin{theorem}
\label{thm:conv-gen}
Let $\nG$ be described by (\ref{eq:def2}) 
and $\cdfgllalt$ as in (\ref{eq:GLL}). If $F\in\IOR$ and $\theta,\,\theta_1\geq 1$, then $\nG\leqc\cdfgllalt$. On the other hand, if $F\in\DOR$ and $\theta,\,\theta_1\leq 1$, then $\cdfgllalt\leqc\nG$.
\end{theorem}
\begin{proof}
Assume that $F\in\IOR$ and $\theta,\,\theta_1\geq1$. Due to the invariance of the convex transform order with respect to scale parameters, we may assume $\beta_1=1$. Hence, we want to prove the convexity of
$$
\psi(x)=K_{\alpha_1,1,\theta_1}^{-1}\circ\nG(x)
=\left(\beta\left(\left(\alpha+\Lambda_F(x)\right)^\theta-\alpha^\theta\right)+\alpha_1\right)^{\theta_1}-\alpha_1^{\theta_1}.
$$
Differentiation shows that
$$\psi^\prime(x)
=\beta\theta\theta_1\left(\beta\left(\left(\alpha+\Lambda_F(x)\right)^\theta-\alpha^\theta\right)+\alpha_1\right)^{\theta_1-1}
     \lambda_F(x)\left(\alpha+\Lambda_F(x)\right)^{\theta-1},
$$
which, under our assumptions, is clearly increasing, so $\psi$ is convex. The second statement is proved analogously.
\end{proof}

\begin{theorem}
\label{thm:disp}
Let $\nG$ be described by (\ref{eq:def2}) (or (\ref{eq:nGbar12}) for a more explicit expression) and $\cdfgllalt$ as in (\ref{eq:GLL}). If $F\in\IOR$, $\theta,\,\theta_1\geq 1$ and $\beta\beta_1\theta\theta_1f(0)\alpha^{\theta-1}\alpha^{\theta_1-1}\geq 1$, then $\nG\leq_{disp}\cdfgllalt$. On the other hand, if $F\in\DOR$, $\theta,\,\theta_1\leq 1$ and $\beta\beta_1\theta\theta_1f(0)\alpha^{\theta-1}\alpha^{\theta_1-1}\leq1$, then $\cdfgllalt\leq_{disp}\nG$.
\end{theorem}
\begin{proof}
The result follows by studying the monotonicity of the function $\phi(x)=K_{\alpha_1,\beta_1,\theta_1}^{-1}\circ\nG(x)-x$. Observe that if $F\in\IOR$ and $\theta, \theta_1 \geq 1$, $\phi^\prime$ is increasing while the additional assumption ensures that $\phi^\prime(0) \geq 0$, establishing the nonnegativeness of $\phi^\prime$. The second part of the theorem follows in a similar manner.
\end{proof}
\begin{remark}
Notice that $\cdfgllalt(0) = \nG(0) = 0$. Thus, under the same conditions as in Theorem \ref{thm:disp}  we can easily get the respective results for the usual stochastic order by applying Theorem 3.B.13(a) of \cite{shaked2007}.
\end{remark}









\bibliography{ModPropOdds}

\end{document}